\documentclass[12pt,reqno]{amsart}
\usepackage{indentfirst, amssymb, amsmath, amsthm, mathrsfs, setspace, indentfirst, enumerate,  mathrsfs, amsmath, amsthm}
\usepackage[bookmarksnumbered, colorlinks, plainpages]{hyperref}
\usepackage{mathrsfs}
\usepackage{cite}
\usepackage{tikz}
\usepackage{graphicx}
\usetikzlibrary{arrows.meta,calc,decorations.pathreplacing,positioning}
\usetikzlibrary{shapes.geometric}
\usepackage{float}
\usepackage{pgfplots}
\usetikzlibrary{shadows}

\usepgfplotslibrary{fillbetween}
\definecolor{curveblue}{RGB}{30,90,220}
\definecolor{negred}{RGB}{255,182,193}
\definecolor{posgreen}{RGB}{40,150,70}
\definecolor{rootdark}{RGB}{120,40,20}
\usepgfplotslibrary{fillbetween}
\definecolor{c1}{RGB}{70,40,200}

\definecolor{c2}{RGB}{213,14,0}

\definecolor{c3}{RGB}{0,158,10}

\definecolor{c4}{RGB}{230,126,34}
\newtheorem*{theoA}{Theorem A}
\newtheorem*{theoB}{Theorem B}
\newtheorem*{theoC}{Theorem C}
\newtheorem*{theoD}{Theorem D}
\newtheorem*{theoE}{Theorem E}
\newtheorem*{theoF}{Theorem F}
\newtheorem*{theoG}{Theorem G}
\newtheorem*{theoH}{Theorem H}

\newtheorem*{BC}{Bieberbach's conjecture}
\newtheorem*{ProA}{Proposition A}

\newtheorem*{cor A}{Corollary A}
\newtheorem*{cor B}{Corollary B}

\newtheorem{theo}{Theorem}[section]
\newtheorem{lem}{Lemma}[section]

\newcommand{\ol}{\overline}
\newcommand{\be}{\begin{equation}}
	\newcommand{\ee}{\end{equation}}
\newcommand{\beas}{\begin{eqnarray*}}
	\newcommand{\eeas}{\end{eqnarray*}}
\newcommand{\bea}{\begin{eqnarray}}
	\newcommand{\eea}{\end{eqnarray}}

\numberwithin{equation}{section}

\begin{document}

\title[O\MakeLowercase {n} \MakeLowercase{Coefficient Theorems for Starlike and Quasi}......]{\LARGE O\Large\MakeLowercase {n} \MakeLowercase{Coefficient Theorems for Starlike and Quasi-Convex of Type} $B$ \MakeLowercase{Mappings in Complex} B\MakeLowercase {anach Spaces}}

\date{}
\author[J. Sok\'o{\l}, S. Majumder, {and} G. Haldar]{Janusz Sok\'o{\l}$^{1}$, Sujoy Majumder$^{2}$ {and}  Goutam Haldar$^{3*}$ }

\address{$^{1}$Faculty of Exact and Technical Sciences, University of Rzesz\'ow, Street: Prof. Pigonia 1, 35-310 Rzesz\'ow, Poland}
\email{jsokol@ur.edu.pl}

\address{$^{2}$Sujoy Majumder, Department of Mathematics, Raiganj University, Raiganj, West Bengal-733134, India.}
\email{sm05math@gmail.com, sjm@raiganjuniversity.ac.in}

\address{$^3$Goutam Haldar, Department of Mathematics, Ghani Khan Choudhury Institute of Engineering and Technology, Narayanpur, Malda-732141, West Bengal, India.}
\email{goutamiitm@gmail.com, goutamiit1986@gmail.com}


\renewcommand{\thefootnote}{}
\footnote{*\emph{Corresponding Author}: Goutam Haldar.}

\renewcommand{\thefootnote}{\arabic{footnote}}
\setcounter{footnote}{0}

\begin{abstract}In this paper, we study biholomorphic mappings $F: \mathbb{B} \to X$ of the form $F(z) = f(z)z$ for $z \in \mathbb{B}$, where $f: \mathbb{B} \to \mathbb{C}$ is a holomorphic function satisfying $f(0) = 1$ and $X$ is a complex Banach space  with the unit ball $\mathbb{B}$.
The main contributions of this paper are threefold. First, we establish a sufficient coefficient condition for a function $F(z)=zf(z)$ to be starlike and quasi-convex of type $B$ in a complex Banach space. Second, we prove that the Bieberbach conjecture holds for these classes of functions. Finally, we analyze the difference of coefficients for such functions in this setting. The results established here directly extend the previous results to the broader setting of complex Banach spaces.
\end{abstract}

\thanks{Typeset by \AmS -\LaTeX}
\maketitle

\section{\bf Introduction}

Let $\mathcal{S}(\mathbb{U})$ denote the class of functions $f$ that are analytic and univalent in the open unit disk $\mathbb{U} = \{z \in \mathbb{C}: \vert{}z\vert{} < 1\}$, normalized by $f(0) = 0$ and $f'(0) = 1$. Any function $f \in \mathcal{S}(\mathbb{U})$ has a power series expansion of the form:
\begin{align}\label{eq1}
\displaystyle
	f(z) = z + \sum_{n=2}^{\infty} a_n z^n, \quad z \in \mathbb{U}.
\end{align}

By employing the Area Theorem, Bieberbach \cite{Bieberbach-1916} established that if $f \in \mathcal{S}(\mathbb{U})$, then $\lvert a_2 \rvert \leq 2$. The following theorem, introduced in 1916, laid the foundation for the celebrated \textbf{Bieberbach conjecture}.

\begin{theoA}{\em{\cite{Bieberbach-1916}}} If $f \in \mathcal{S}(\mathbb{U})$, then $|a_2|\leq 2$. Equality holds if, and
only if, $f$ is a rotation of the Koebe function.
\end{theoA}

 Bieberbach's conjecture served as the catalyst for the study of univalent functions, driving its development over the last century into the extensive field of mathematical research we see today.
\begin{BC} If $\in\mathcal{S}(\mathbb{U})$ and is given by \eqref{eq1}, then $|a_n|\leq n$ for any integer $n\geq 2$. Equality holds if and only if $f$ is a rotation of the Koebe function.
\end{BC}

This conjecture evolved into one of the most celebrated problems in twentieth-century mathematics. Progress, however, was gradual: by 1968, it had been verified only for $n = 2, 3, 4, 5,$ and $6$. For $n \geqslant 3$, the proofs demanded sophisticated, non-elementary techniques. The case $n = 3$ was settled in 1923 by Loewner \cite{Loewner-1923}, who developed an ingenious approach based on a differential equation whose solutions form a dense subset of $\mathcal{S}(\mathbb{U})$. Loewner’s parametric method not only resolved the conjecture for $n = 3$, but was also used to determine sharp coefficient bounds for the inverse functions of $f \in \mathcal{S}(\mathbb{U})$ (see, for example, Hayman \cite{Hayman-1994}). Crucially, this theory served as a cornerstone decades later when de Branges completed his celebrated proof of the full Bieberbach conjecture \cite{Branges-1985}.\vspace{1.2mm}

Throughout the twentieth century, research also focused on finding the smallest constant $C$ such that $\vert{}a_n\vert{} < C n$. By 1976, the best-known bound was $C \approx 1.0691$, which remained unimproved for general $n \geqslant 2$ until de Branges proved the full Bieberbach conjecture in 1984.\vspace{1.2mm}

A normalized univalent function $f \in \mathcal{S}(\mathbb{U})$ is said to be starlike (with respect to the origin) if and only if
\begin{align*}
\Re\left\lbrace \frac{z f'(z)}{f(z)} \right\rbrace> 0
\end{align*}
 for all $z \in \mathbb{U}$. The subclass consisting of all such starlike functions is denoted by $\mathcal{S}^*(\mathbb{U})$.

 Let $\mathcal{K}(\mathbb{U})$ denote the familiar subclass of $\mathcal{S}(\mathbb{U})$ consisting of functions which are convex in $\mathbb{U}$. Then $f\in\mathcal{K}(\mathbb{U})$ if and only if
 $\Re\left\lbrace1+\frac{z f''(z)}{f'(z)}\right\rbrace >0$, $z \in \mathbb{U}$.

We begin by recalling a simple sufficient condition on the coefficients of $f(z)$ for starlikeness, due to Clunie and Keogh \cite{Clunie-1960}.

\begin{theoB}{\em{\cite{Clunie-1960}}} Let $f \in \mathcal{S}(\mathbb{U})$ with $f(z) = z + \sum_{n=2}^{\infty} a_n z^n$. If $\sum_{n=2}^{\infty} n \vert{}a_n\vert{} \leqslant 1,$ then $f \in \mathcal{S}^*(\mathbb{U})$.
\end{theoB}

We next consider the Bieberbach conjecture for the class $\mathcal{S}^*(\mathbb{U})$ of starlike functions. R. Nevanlinna \cite{Nevanlinna-1920} (see also \cite[Theorem 2.14]{Duren-1983}) first established this result in 1920.

\begin{theoC}{\em{\cite{Nevanlinna-1920}}} The coefficients of each function $f\in\mathcal{S}^*(\mathbb{U})$ satisfy $|a_n|\leq n$ for $n=2,3,\ldots$. Strict inequality holds for all $n$ unless $f$ is a rotation of the Koebe function.
\end{theoC}

\begin{theoD}{\em{\cite[Corollary, pp. 45]{Duren-1983}}} If $f\in\mathcal{K}(\mathbb{U})$, then $|a_n|\leq 1$ for $n=2,3,\ldots$. Strict inequality holds for all  $n$ unless $f$ is a rotation of the function $\mathcal{L}$ defined by $\mathcal{L}(z)=z(1-z)^{-1}$.
\end{theoD}

In addition to the inequality $\vert{}a_n\vert{} \leqslant 1$ for $n \geqslant 2$ mentioned in Theorem D, Theorem B and Alexander's theorem together yield the sufficient condition
$\sum_{n=2}^{\infty} n^2 \vert{}a_n\vert{} \leqslant 1$
 for a function $f$ to belong to $\mathcal{K}(\mathbb{U})$ (see also \cite[p. 60]{Thomas-Tuneski}).\vspace{1.2mm}

 A more general and elegant inequality concerning the coefficients of the function $f\in\mathcal{S}^*(\mathbb{U})$ can be stated as follows
\begin{theoE}{\em{\cite[Theorem 4.1.2]{Thomas-Tuneski}}} Let $f \in \mathcal{S}^*(\mathbb{U})$ and be given by \eqref{eq1}. Then
\begin{align*}
(n + 1)^2 \vert{}a_n\vert{}^2 \leqslant 4 \left( 1 + \sum_{k=2}^{n} k \vert{}a_k\vert{}^2 \right)
\end{align*}
for $n \geqslant 2$. Equality holds if and only if $f$ is a rotation of the Koebe function.
\end{theoE}

We now present an analogue of Theorem E for the class $f\in\mathcal{K}(\mathbb{U})$.
\begin{theoF}{\em{\cite[Theorem 4.2.4]{Thomas-Tuneski}}}
Let $f \in \mathcal{K}(\mathbb{U})$ be given by \eqref{eq1}. Then for $n \geqslant 2$,
\begin{align*}
n^2(n + 1)^2 \vert{}a_n\vert{}^2 \leqslant 4 \left( 1 + \sum_{k=2}^{n} k^3 \vert{}a_k\vert{}^2 \right).
\end{align*}
Equality holds for the convex function $f(z)=z(1 - z)^{-1}$.
\end{theoF}
\vspace{1.2mm}

Another coefficient problem which has attracted considerable attention is to estimate
\begin{align*}
d_n=|\;|a_{n+1}|-|a_n|\;|,\quad n=2,3,\ldots,
\end{align*}
the difference of the moduli of successive coefficients of a function $f\in\mathcal{S}(\mathbb{U})$.\vspace{1.2mm}

Finding the best upper bound for the difference of coefficients $d_n$ when
$f \in \mathcal{S}(\mathbb{U})$ remains open. Hayman \cite{Hayman-1963} showed that $d_n\le A$ for $n \ge 2$, where $A$ is an absolute constant. Milin \cite{Milin-1968}, using a different method, showed that $A < 9$; Ilina \cite{Ilina-1968} improved this to $A < 4.26\dots$, and Grinšpan \cite{Grinspan} to $A < 3.61\dots$.

On the other hand, for $n = 2$, Goluzin \cite{Goluzin-1946,Goluzin-1949}, Jenkins \cite{Jenkins-1960}, and Duren \cite[p. 114]{Duren-1983} proved reminiscent of the sharp bound in the Fekete-–Szeg\"{o} theorem that for $f \in \mathcal{S}(\mathbb{U})$, the following sharp inequality holds:
\begin{align*}
-1 \le |a_3| - |a_2| \le \frac{3}{4} + e^{-\lambda_0}\left(2e^{-\lambda_0} - 1\right) = 1.029\dots,
\end{align*}
where $\lambda_0$ is the unique value in $(0, 1)$ satisfying $4\lambda = e^\lambda$.

Thus, for general functions $f \in \mathcal{S}$, the constant $A$ cannot be reduced to $1$. However, for starlike functions, Leung \cite{Leung-1978} used the third Lebedev--Milin inequality to prove that the best possible bound is $A = 1$. We present his result in the following theorem (see also \cite[Thorem 4.1.4]{Thomas-Tuneski}).

\begin{theoG} Let $f \in \mathcal{S}^*(\mathbb{U})$ and be given by \eqref{eq1}. Then, $\vert{}\vert{}a_{n+1}\vert{} - \vert{}a_n\vert{}\vert{} \le 1$  for $n \ge 2$. For a fixed $n$, equality holds if and only if $f(z) = f_0(z)$, where
\begin{align*}
f_0(z) = \frac{z}{(1 - \gamma z)(1 - \zeta z)}
\end{align*}
for some $\gamma$ and $\zeta$ satisfying $\vert{}\gamma\vert{} = \vert{}\zeta\vert{} = 1$.
\end{theoG}

For $f \in \mathcal{S}^*(\mathbb{U})$, Theorem G establishes that $A = 1$ for $n \ge 2$. A natural question that arises is whether $A=1$ also holds when $f \in \mathcal{K}(\mathbb{U})$. The following result is due to
Koepf \cite{Koepf-1987} (see also \cite[Thorem 9.3.7]{Thomas-Tuneski}).
\begin{theoH} Let $f\in\mathcal{K}(\mathbb{U})$ and be given by \eqref{eq1}. Then $||a_3|-|a_2||\leq 1$.
The inequality is sharp for the Koebe function.
\end{theoH}

Before going further let us define some terms which will recur.
Let $X$ a complex Banach space with norm $||.||$. The ball of radius $r$, $\mathbb{B}_r =\{z\in X:||z||<r\}$. If $r=1$, we will simply use $\mathbb{B}$ and if $X=\mathbb{C}$, then $\mathbb{B}=\mathbb{U}$. Let $\partial \mathbb{B}$ denote the boundary of $\mathbb{B}$ and $\overline{\mathbb{B}}$ its closure.\vspace{1.2mm}

Now let $X$ and $Y$ be two complex Banach spaces with respect to the norms $||.||_1$ and $||.||_2$ respectively. For simplicity, we denote both norms by $||.||$, when there is no possibility of confusion. Let $L(X, Y)$ denote the set of continuous linear operators from a complex Banach space $X$ into a complex Banach space $Y$ with the standard operator norm 
\begin{align*} 
||A||=\sup\{||A(z)||: ||z||=1,\quad A\in L(X,Y)\}.\end{align*} Let $I$ be the identity operator in $L(X, X)$.  For each $z \in X \setminus \{0\}$, define the set
\begin{align*}
\displaystyle
T(z) = \{ l_z \in L(X, \mathbb{C}) : l_z(z) = \Vert{}z\Vert{}, \quad \Vert{}l_z\Vert{} = 1 \}.
\end{align*}
This set is guaranteed to be nonempty by the Hahn-Banach theorem. Let $\Omega$ be a domain in $X$ and let $f:\Omega \rightarrow Y$ be a mapping. Then $f$ is called holomorphic on $\Omega$ if for any $z\in\Omega$ there is a mapping $Df(z)\in L(X,Y)$, called the Fr\'echet derivative of $f$ at $z$ such that \begin{align*} \lim_{h\to 0} \frac{||f(z+h)-f(z)-Df(z)h||}{||h||}=0,
\end{align*} 
i.e., 
\begin{align*} 
f(z+h)=f(z)+Df(z)h+o(||h||).
\end{align*}
Let $\mathcal{H}(\Omega, \Omega')$ represent the set of all holomorphic mappings from a domain $\Omega \subset X$ into another domain $\Omega' \subset Y$. We simply write $\mathcal{H}(\Omega, X) = \mathcal{H}(\Omega)$. It is well known that if $f \in \mathcal{H}(\Omega,Y)$, then for all $w$ in some neighborhood of $z \in \Omega$, the function can be expressed as the power series expansion:
\begin{align*}
\displaystyle
f(w) = \sum_{n=0}^{\infty} \frac{1}{n!} D^n f(z)((w - z)^n),
\end{align*}
where $D^n f(z)$ is a bounded symmetric $n$-linear mapping from $\prod_{j=1}^n X$ into $Y$, called the $n$-th Fr\'echet derivative of $f$ at $z$. Moreover, the term $D^n f(z)((w - z)^n)$ is shorthand for applying the $n$-linear mapping to $n$ copies of the vector $(w - z)$:
\begin{align*}
\displaystyle
D^n f(z)((w - z)^n) = D^n f(z)(\underbrace{w - z, w - z, \dots, w - z}_{n \text{ times}}),
\end{align*}
where $D^0f(z)=f(z)$.\vspace{1.2mm}

Let $\mathbb{B}$ be the open unit ball in a complex Banach space $X$. A mapping $f \in \mathcal{H}(\mathbb{B})$ is said to be normalized if $f(0) = 0$ and $Df(0) = I$, where $I \in L(X,X)$ is the identity operator. It is called biholomorphic if $f$ is bijective with a holomorphic inverse $f^{-1}$ on $f(\mathbb{B})$, and locally biholomorphic if $Df(z)$ has a bounded inverse for every $z \in \mathbb{B}$. The family of all normalized biholomorphic mappings on $\mathbb{B}$ is denoted by $\mathcal{S}(\mathbb{B})$. \vspace{1.2mm}

In several complex variables, classical results from geometric function theory often fail to hold in their original form. Most notably, Poincar\'e disproved the higher-dimensional analogue of the Riemann mapping theorem, while Cartan constructed counterexamples to the Bieberbach conjecture on the unit ball.

As a result, extending univalent function theory to higher dimensions requires introducing appropriate subclasses or additional structural restrictions. Recently, coefficient bounds for various subclasses of normalized biholomorphic mappings in higher dimensions have been established by Bracci et al. \cite{Bracci-2016},
Bracci and Roth \cite{Bracci-2018}, Graham et al. \cite{Graham-2002}, \cite{Graham-2016}, Hamada and Honda \cite{Hamada-Honda-2008} and Hamada and Kohr \cite{Hamada-Kohr-2020}.\vspace{2mm}

In this paper, we mainly consider the biholomorphic mappings $F:\mathbb{B}\to X$ of the form $F(z)=f(z)z$ for $z\in\mathbb{B}$, where $f:\mathbb{B}\to \mathbb{C}$ is a holomorphic function satisfying $f(0)=1$.
The contributions of this paper are threefold. First, we establish a sufficient coefficient condition for a function $f$ to be starlike as well as quasi-convex of type $B$ in a complex Banach space. Second, we prove that the Bieberbach conjecture holds for such functions. Finally, we investigate the difference of coefficients for starlike and quasi-convex functions of type $B$ in this setting. These results extend the above Theorems A–-H to higher dimensions.

\section{{\bf Coefficient theorems of starlike mappings in complex Banach space}}\label{Sec-2}

A biholomorphic mapping $f$ on $\mathbb{B}$ is said to be starlike if $f(\mathbb{B})$ is starlike with
respect to the origin. If $\mathbb{B}$ is the unit ball of $X$, a complex Banach space, we let $\mathcal{S}^*(\mathbb{B})$ denote the subset of $\mathcal{S}(\mathbb{B})$ consisting of normalized starlike mappings from $\mathbb{B}$ into $X$.

The following proposition is the well-known criterion for starlikeness on the unit
ball in a complex Banach space.
\begin{ProA}{\em{\cite{Su2}}} Let $f:\mathbb{B}\to  X$ be a locally biholomorphic mapping. Then $f$ is a starlike mapping on $\mathbb{B}$ if and only if
\begin{align}
\Re\left\lbrace l_z\left([Df(z)]^{-1}f(z)\right)\right\rbrace>0,\quad z\in\mathbb{B}\backslash\{0\},\quad l_z\in T(z).
\end{align}
\end{ProA}

 There are now many references for starlike mappings in higher dimensions. For more information, see the recent monographs by Gong  \cite{Gong} and Kohr \& Liczberski  \cite{Kohr- Liczberski-1998}. Some of the foundational and original papers in this area include
 \cite{Chu}, \cite{SSQ_CHMS_1990}, \cite{Su2}-\cite{Xu-He-Xu}.
 \vspace{1.1mm}

 By considering the biholomorphic mappings $F:\mathbb{B}\to X$ of the form $F(z)=f(z)z$ for $z\in\mathbb{B}$, where $f:\mathbb{B}\to \mathbb{C}$ is a holomorphic function satisfying $f(0)=1$, we begin by giving a sufficient condition concerning the coefficients of  a function $F$ to be starlike in the complex Banach space.
 \begin{theo}\label{Th-21}Let $f\in\mathcal{H}(\mathbb{B},\mathbb{C})$ with $f(0)=1$ and $F(z)=zf(z)\in\mathcal{S}(\mathbb{B})$. If
 \begin{align*}	
 \sum\limits_{n=2}^{\infty}n\left|l_{z}\left(\frac{D^{n}F(0)(z^{n})}{n!||z||^n}\right)\right|\leq 1
\end{align*}	
 for $z\in\mathbb{B}\backslash\{0\}$ and $l_z\in T(z)$, then $F(z)=zf(z)\in\mathcal{S}^*(\mathbb{B})$.	
 \end{theo}
 In the next result, we show that Bieberbach conjecture holds for the function $F(z)=zf(z)$ in $\mathcal{S}^*(\mathbb{B})$.
 \begin{theo}\label{Th-22}Let $f\in\mathcal{H}(\mathbb{B},\mathbb{C})$ with $f(0)=1$ and $F(z)=zf(z)\in\mathcal{S}^*(\mathbb{B})$. 
 Then for $z\in\mathbb{B}\backslash\{0\}$ and $l_z\in T(z)$, we have
 \begin{align*}
 	\left|l_{z}\left(\frac{D^{k}F(0)(z^{k})}{k!||z||^k}\right)\right|\leq k
 \end{align*}
 and
 \begin{align}\label{Th-22t}
 (n+1)^2\left|l_{z}\left(\frac{D^{n}F(0)(z^{n})}{n!||z||^n}\right)\right|^2\leq4\left(1+\sum\limits_{k=2}^{n}k	\left|l_{z}\left(\frac{D^{k}F(0)(z^{k})}{k!||z||^k}\right)\right|^2\right).	
 \end{align}
 Equality holds for the function $F(z)=f(z)z$, where
 \begin{align*}
 	f(z)=\frac{1}{(1-l_u(z))^2},\quad z\in\mathbb{B}\setminus\{0\},\quad ||u||=1.
 \end{align*}
 \end{theo}

 The next result is related to difference of coefficients for the function $F(z)=zf(z)$ in $\mathcal{S}^*(\mathbb{B})$.
 
 \begin{theo}\label{Th-23}Let $f\in\mathcal{H}(\mathbb{B},\mathbb{C})$ with $f(0)=1$ and $F(z)=zf(z)\in\mathcal{S}^*(\mathbb{B})$. Then for $k\geq2$, we have
\begin{align}\label{Th-23t}
\displaystyle
&\left|\;\left|\frac{l_z\left(D^{k+1}F(0)\left(z^{k+1}\right)\right)}{(k+1)!||z||^{k+1}}\right|-\left|\frac{l_z\left(D^{k}F(0)\left(z^{k}\right)\right)}{k!||z||^{k}} \right|\;\right|\leq 1,\quad z\in\mathbb{B}\backslash\{0\},\quad l_z\in T(z).
\end{align}
Equality holds for the function $F$ defined as in Theorem \ref{Th-22}.
 \end{theo}

 \subsection{{Auxiliary lemma}}
 
 Let $\mathcal{P}$ denote the class of analytic functions $p$ in $\mathbb{U}$ satisfying $p(0) = 1$ and $\Re\{p(z)\}> 0$ for all $z \in \mathbb{U}$. Member of $\mathcal{P}$ is called a Carath\'eodory function.  The lemma below is instrumental in establishing our main results.
\begin{lem}\label{Lem1} {\em{\cite{Cho-Kumar-Ravichandran-2019}}}
Let $p(z)=1+\sum_{n=1}^{\infty}p_nz^n\in\mathcal{P}$. Then
\begin{align*}
\displaystyle
|p_n|\leq 2\quad \text{and}\quad |p_n-p_mp_{n-m}|\leq 2,\quad n\geq 2,\;1\leq m\leq n-1.
\end{align*}
\end{lem}
\begin{proof}[\bf Proof of Theorem \ref{Th-21}]
By the given condition, we have $f\in\mathcal{H}(\mathbb{B},\mathbb{C})$ with $f(0)=1$ and $F(z)=zf(z)\in\mathcal{S}(\mathbb{B})$. Let $z\in \mathbb{B}\backslash\{0\}$ be fixed. Suppose $z_0=z/\|z\|$. We define the function $g:\mathbb{U}\to \mathbb{C}$ as follows
\begin{align}\label{Th-21.0}	
	\displaystyle
	g(\zeta)=\begin{cases}
		\frac{\zeta}{l_{z_0}\left([D F(\zeta z_0)]^{-1}F(\zeta z_0)\right)},& \zeta\neq 0,\\[8pt]
		1& \zeta=0.
	\end{cases}
\end{align}
Using a similar method as in Theorem 7.1.14 \cite{Graham-Kohr}, we have
\begin{align}\label{Th-21.1}
	[D F(z)]^{-1}\eta=\frac{1}{f(z)}\left\lbrack \eta-\frac{\frac{(Df(z){\eta})z}{f(z)}}{1+\frac{Df(z)z}{f(z)}}\right\rbrack.
\end{align}
Therefore,
\begin{align}\label{Th-21.2}
	[D F(z)]^{-1}F(z)=\frac{zf(z)}{f(z)+D f(z)z},
\end{align}
which implies that
\begin{align}\label{Th-21.3}
g(\zeta)=\frac{\zeta}{l_{z_0}\left([D F(\zeta z_0)]^{-1}F(\zeta z_0)\right)}=1+\frac{Df(\zeta z_0) \zeta z_0}{f(\zeta z_0)}.
\end{align}
If
\begin{align}\label{Th-21.4}
f(z)=1+\sum\limits_{k=1}^{\infty} \frac{1}{k!}D^kf(0)(z^k),
\end{align}	
then it is clearly
\begin{align}\label{Th-21.5}
Df(z)(z)=\sum\limits_{k=1}^{\infty} \frac{k}{k!}D^kf(0)(z^k).
\end{align}	
Since $F(z)=zf(z)$, we have (see \cite[Theorem 2]{Xiong-Yu-Wang-2025})
\begin{align}\label{Th-21.6}
D^mF(0)(z^m)=mD^{m-1}f(0)(z^{m-1})z
\end{align}
where $m\in\mathbb{N}$. As $z_0=z/||z||$, we deduce from \eqref{Th-21.6} that
\begin{align}\label{Th-21.7}
\displaystyle
\frac{l_{z_0}(D^mF(0)(z_0^m))}{m!}=\frac{D^{m-1}f(0)(z_0^{m-1})}{(m-1)!}.
\end{align}
On the other hand, using \eqref{Th-21.7} to \eqref{Th-21.4} and \eqref{Th-21.5}, we find that
\begin{align}\label{Th-21.8}
f(\zeta z_0)=1+\sum\limits_{k=1}^{\infty} l_{z_0}\left(\frac{1}{(k+1)!}D^{k+1}F(0)(z_0^{k+1})\right)\zeta^k
\end{align}	
and
\begin{align}\label{Th-21.9}
Df(\zeta z_0)(\zeta z_0)=\sum\limits_{k=1}^{\infty} kl_{z_0}\left(\frac{1}{(k+1)!}D^{k+1}F(0)(z_0^{k+1})\right)\zeta^{k}.
\end{align}	
Now using \eqref{Th-21.8} and \eqref{Th-21.9}, we deduce that
\begin{align*}
&|g(\zeta)f(\zeta z_0)-f(\zeta z_0)|-|f(\zeta z_0)|\\=&|Df(\zeta z_0))(\zeta z_0)|-|f(\zeta z_0)|\nonumber\\=&
\left|\sum\limits_{k=1}^{\infty}kl_{z_0}\left(\frac{D^{k+1}F(0)(z_0^{k+1})}{(k+1)!}\right)\zeta^{k}\right|-\left|1+\sum\limits_{k=1}^{\infty}l_{z_0}\left(\frac{D^{k+1}F(0)(z_0^{k+1})}{(k+1)!}\right)\zeta^{k}\right|\nonumber\\<&
\sum\limits_{k=1}^{\infty}k\left|l_{z_0}\left(\frac{D^{k+1}F(0)(z_0^{k+1})}{(k+1)!}\right)\right|-\left(1-\sum\limits_{k=1}^{\infty}\left|l_{z_0}\left(\frac{D^{k+1}F(0)(z_0^{k+1})}{(k+1)!}\right)\right|\right)\nonumber\\=&
\sum\limits_{n=2}^{\infty}n\left|l_{z_0}\left(\frac{D^{n}F(0)(z_0^{n})}{n!}\right)\right|-1\leq 0\nonumber
\end{align*}
and so for $\zeta\in\mathbb{U}\backslash\{0\}\}$, we have $\left|g(\zeta)-1\right|<1.$ This shows that  $\Re\{g(\zeta)\}>0$ for all $\zeta\in\mathbb{U}\backslash\{0\}$, i.e.,
\begin{align}\label{Th-21.10}
\Re\left\lbrace \frac{\zeta}{l_{z_0}\left([D F(\zeta z_0)]^{-1}F(\zeta z_0)\right)}\right\rbrace>0
\end{align}
for all $\zeta\in\mathbb{U}\backslash\{0\}$. Substituting $\zeta=||z||$ into \eqref{Th-21.10} and using the fact that
$l_{\alpha z}=\frac{|\alpha|}{\alpha}l_z$
for each $\alpha\in\mathbb{U}\backslash\{0\}$, we obtain
\begin{align*}
\Re\left\lbrace \frac{||z||}{l_{z}\left([D F(z)]^{-1}F( z)\right)}\right\rbrace>0,\quad z\in\mathbb{B}\backslash\{0\}, \quad l_z\in T(z).
\end{align*}
Consequently $F\in \mathcal{S}^*(\mathbb{B})$.
\end{proof}

\begin{proof}[\bf Proof of Theorem \ref{Th-22}]
By the given condition, $F\in\mathcal{S}^*(\mathbb{B})$. Then we have
\begin{align}\label{Th-22.0}
	\displaystyle
	\Re\left\lbrace \frac{\|z\|}{l_z\left([D F(z)]^{-1}F(z)\right)}\right\rbrace>0, \quad z\in\mathbb{B}\backslash \{0\},\quad l_z\in T(z).
\end{align}
Let $z\in \mathbb{B}\backslash\{0\}$ be fixed. Suppose $z_0=z/\|z\|$. Also, we consider the function $g:\mathbb{U}\to \mathbb{C}$ defined by \eqref{Th-21.0}.
Clearly $g\in\mathcal{H}(\mathbb{U})$ and $g(0)=1$.
 Therefore for $\zeta\in\mathbb{U}\backslash\{0\}$, we have
\begin{align}\label{Th-22.1}
	\displaystyle
	g(\zeta)=\frac{\zeta}{l_{z_0}\left([D F(\zeta z_0)]^{-1}F(\zeta z_0)\right)}=\frac{\|\zeta z_0\|}{l_{\zeta z_0}\left([D F(\zeta z_0)]^{-1}F(\zeta z_0)\right)}.
\end{align}	
Consequently $\Re\{g(\zeta)\}>0$ for all $\zeta\in\mathbb{U}$. Suppose
\begin{align}\label{Th-22.2}
g(\zeta)=1+\sum_{k=1}^{\infty}g_k\zeta^k.
\end{align}

Using \eqref{Th-21.8}, \eqref{Th-21.9} and \eqref{Th-22.2} to \eqref{Th-21.3}, we get
\begin{align}\label{Th-22.3}
&\left(1+\sum_{k=1}^{\infty}g_k\zeta^k\right)\left(1+\sum\limits_{k=1}^{\infty}l_{z_0}\left(\frac{1}{(k+1)!}D^{k+1}F(0)(z_0^{k+1})\right)
\zeta^{k}\right)\\=&
1+\sum\limits_{k=1}^{\infty} (k+1)l_{z_0}\left(\frac{1}{(k+1)!}D^{k+1}F(0)(z_0^{k+1})\right)\zeta^{k}.\nonumber	
\end{align}
By equating coefficients on the both sides of \eqref{Th-22.3}, we get
\begin{align}\label{Th-22.4}
(k+1)l_{z_0}\left(\frac{1}{(k+1)!}D^{k+1}F(0)(z_0^{k+1})\right)=&l_{z_0}\left(\frac{1}{(k+1)!}D^{k+1}F(0)(z_0^{k+1})\right)\\+&\sum\limits_{i=1}^{k-1}g_{k-i}\;l_{z_0}\left(\frac{1}{(i+1)!}D^{i+1}F(0)(z_0^{i+1})\right)+g_{k}.\nonumber
\end{align}
We now use mathematical induction and suppose
\begin{align*}
\left|l_{z_0}\left(\frac{1}{i!}D^{i}F(0)(z_0^{i})\right)\right|\leq i\quad \text{for}\;\;i=2,3,\ldots,k.
\end{align*}
Applying Lemma \ref{Lem1} to \eqref{Th-22.4}, we get
\begin{align*}
k\left|l_{z_0}\left(\frac{1}{(k+1)!}D^{k+1}F(0)(z_0^{k+1})\right)\right|\leq 2\sum\limits_{i=1}^{k-1}(i+1)+2=k(k+1),
\end{align*}
which shows that
\begin{align*}
\left|l_{z_0}\left(\frac{1}{(k+1)!}D^{k+1}F(0)(z_0^{k+1})\right)\right|\leq k+1.
\end{align*}
Then by induction, we have
\begin{align*}
\left|l_{z_0}\left(\frac{1}{k!}D^{k}F(0)(z_0^{k})\right)\right|\leq k.
\end{align*}
\par This proves the first inequality.
\par Since $\Re\{g(\zeta)\}>0$ for all $\zeta\in\mathbb{U}$, there exists a Schwarz function $\omega(\zeta)$ with $\omega(\zeta)=\sum_{k=1}^{\infty}\omega_k\zeta^k$ such that $g(\zeta)=(1+\omega(\zeta))/(1-\omega(\zeta))$ and so from \eqref{Th-21.3}, we get
\begin{align}\label{Th-22.5}
\left(2f(\zeta z_0)+Df(\zeta z_0) \zeta z_0\right)\omega(\zeta)=Df(\zeta z_0) \zeta z_0.
\end{align}
Using \eqref{Th-21.8} and \eqref{Th-21.9} to \eqref{Th-22.5}, we obtain
\begin{align}\label{Th-22.6}
\left(2+\sum\limits_{k=1}^{\infty}(k+2)a_{k+1} \zeta^k \right)\omega(\zeta)=\sum\limits_{k=1}^{\infty}ka_{k+1}\zeta^k,
\end{align}
where
\begin{align}\label{Th-22.7}
a_{k+1}=l_{z_0}\left(\frac{1}{(k+1)!}D^{k+1}F(0)(z_0^{k+1})\right)
\end{align}
Equating coefficients on the both sides of \eqref{Th-22.6}, one can easily get
\begin{align*}
2\omega_k+3a_2\omega_{k-1}+\ldots+(k+1)a_k\omega_1=ka_{k+1}.
\end{align*}
This shows that the coefficient $a_{k+1}$ depends only on $a_2$, $a_3\ldots$,$a_k$. Therefore we can rewrite \eqref{Th-22.6} as follows
\begin{align}\label{Th-22.8}
	\left(2+\sum\limits_{k=1}^{n-1}(k+2)a_{k+1} \zeta^k \right)\omega(\zeta)=\sum\limits_{k=1}^{n}ka_{k+1}\zeta^k+\sum\limits_{k=n+1}^{\infty}b_k\zeta^k
\end{align}
for some $b_k$.
Let $\zeta=re^{i\theta}$. Then from \eqref{Th-22.8}, we get
\begin{align}\label{Th-22.9}
\left(2+\sum\limits_{k=1}^{n-1}(k+2)a_{k+1} r^ke^{ik\theta} \right)\omega(re^{i\theta})=\sum\limits_{k=1}^{n}ka_{k+1}r^ke^{ik\theta}+\sum\limits_{k=n+1}^{\infty}b_kr^ke^{ik\theta}
\end{align}
and
\begin{align}\label{Th-22.10}
\left(2+\sum\limits_{k=1}^{n-1}(k+2)\ol a_{k+1} r^ke^{-ik\theta} \right)\omega(re^{-i\theta})=\sum\limits_{k=1}^{n}k\ol a_{k+1}r^ke^{-ik\theta}+\sum\limits_{k=n+1}^{\infty}\ol b_kr^ke^{-ik\theta}.
\end{align}
Note that $|\omega(\zeta)|<1$ for all $\zeta\in\mathbb{U}$. We first multiply \eqref{Th-22.9} and \eqref{Th-22.10} and then integrating with respect to $\theta$ from $0$ to $2\pi$, we can find that
\begin{align*}
\sum\limits_{k=1}^{n}k^2|a_{k+1}|^2r^{2k}+\sum\limits_{k=n+1}^{\infty}|b_k|^2r^{2k}<4+\sum\limits_{k=1}^{n-1}(k+2)^2|a_{k+1}|^2r^{2k}.
\end{align*}
Letting $r\to 1$, we obtain
\begin{align*}
\sum\limits_{k=1}^{n}k^2|a_{k+1}|^2\leq 4+\sum\limits_{k=1}^{n-1}(k+2)^2|a_{k+1}|^2,
\end{align*}
i.e.,
\begin{align*}
n^2|a_{n+1}|^2\leq 4\left(1+\sum\limits_{k=1}^{n-1}(k+1)|a_{k+1}|^2\right)=4\left(1+\sum\limits_{i=2}^{n}i|a_{i}|^2\right),
\end{align*}
or
\begin{align*}
(m-1)^2|a_{m}|^2\leq 4\left(1+\sum\limits_{i=2}^{m-1}i|a_{i}|^2\right).
\end{align*}
Now adding $4m|a_{m}|^2$ on the both sides and then using \eqref{Th-22.7}, we get the required conclusion \eqref{Th-22t}.\vspace{1.9mm}

{\bf Sharpness part:} Let us consider $F(z)=f(z)z$, where
\begin{align*}
f(z)=\frac{1}{(1-l_u(z))^2},\quad z\in\mathbb{B},\quad ||u||=1.
\end{align*}

According to Lemma 2.2 \cite{Xu-He-Xu}, it is clear that $F\in\mathcal{S}^*(\mathbb{B})$. Note that $F(z)=\sum_{n=0}^{\infty}(n+1)(l_u(z))^{n}\;z$, $z\in\mathbb{B}$.
A straightforward computation shows that
\begin{align}\label{Th-22.11}
\frac{D^{k+1}F(0)\left(z^{k+1}\right)}{(k+1)!}=(k+1)(l_u(z))^k z\quad \text{and}\quad \frac{D^{k}F(0)\left(z^{k}\right)}{k!}=k(l_u(z))^{k-1} z.
\end{align}
Taking $z=ru\;(0<r<1)$, we get
\begin{align*}
\frac{l_z\left(D^{k+1}F(0)\left(z^{k+1}\right)\right)}{(k+1)!||z||^{k+1}}=k+1,\quad \frac{D^{k}F(0)\left(z^{k}\right)}{k!||z||^k}=k
\end{align*}
and so equality of \eqref{Th-22t} holds.
\end{proof}	

\begin{proof}[\bf Proof of Theorem \ref{Th-23}]
By the given condition, $F\in\mathcal{S}^*(\mathbb{B})$. Let $z\in \mathbb{B}\backslash\{0\}$ be fixed. Suppose $z_0=z/\|z\|$. We consider the function $g:\mathbb{U}\to \mathbb{C}$ defined by \eqref{Th-21.0}. Clearly
$g\in\mathcal{H}(\mathbb{U})$ with $g(0)=1$ and $\Re\{g(\zeta)\}>0$ for all $\zeta\in\mathbb{U}$.
Let
\begin{align}\label{Th-23.1}
\displaystyle
h(\zeta)=f(\zeta z_0)=1+D f(0)(z_0)\zeta+\frac{1}{2!}D^2 f(0)(z_0^2)\zeta^2+\ldots.
\end{align}
Clearly
\begin{align}\label{Th-23.2}
\displaystyle
D f(\zeta z_0)(\zeta z_0)=D f(0)(z_0)\zeta+D^2 f(0)(z_0^2)\zeta^2+\ldots
\end{align}
From \eqref{Th-23.1} and \eqref{Th-23.2}, we have $h'(\zeta)\zeta=D f(\zeta z_0)(\zeta z_0)$. Now from \eqref{Th-21.3}, we get
\begin{align*}
	g(\zeta)=1+\frac{\zeta h'(\zeta)}{h(\zeta)}.
\end{align*}
Now from \eqref{Th-22.2}, we have
\begin{align*}
h'(\zeta)/h(\zeta)=\sum_{k=1}^{\infty}g_k\zeta^{k-1}.
\end{align*}
Integration gives
\begin{align*}
	\displaystyle
	\log h(\zeta)=\sum\limits_{k=1}^{\infty}\frac{1}{k}g_k\zeta^k.
\end{align*}
Also for $\xi \in\partial \mathbb{U}$, we have
\begin{align}\label{Th-23.3}
	\log \left\lbrace (1-\xi \zeta)h(\zeta)\right\rbrace=\sum\limits_{k=1}^{\infty}\alpha_k\zeta^k,
\end{align}
where
\begin{align}\label{Th-23.4}
	\alpha_k=\frac{1}{k}(g_k-\xi^k).
\end{align}
On the other hand, from \eqref{Th-23.1}, we get
\begin{align}\label{Th-23.5}
(1-\xi \zeta)h(\zeta)=\sum\limits_{k=0}^{\infty}\beta_k\zeta^k,
\end{align}
where
\begin{align}\label{Th-23.6}
	\begin{cases}
		\beta_0=&1,\\[8pt]
		\beta_1=&Df(0)(z_0)-\xi,\\[8pt] \beta_k=&\frac{1}{(k+1)!}D^{k+1}f(0)\left(z_0^{k+1}\right)-\xi\frac{1}{k!} D^kf(0)\left(z_0^k\right)\quad \text{for} \;\;k\geq 2.
	\end{cases}
\end{align}
Now from \eqref{Th-23.3} and \eqref{Th-23.5}, we obtain
\begin{align*}
	\sum\limits_{k=0}^{\infty}\beta_k \zeta^k=\exp\left\lbrace \sum\limits_{k=1}^{\infty} \alpha_k \zeta^k\right\rbrace.
\end{align*}
Applying Third Lebedev-Milin inequality (see \cite[pp.143]{Duren-1983}), we get
\begin{align*}
	|\beta_k|^2\leq \exp\left\lbrace \sum\limits_{n=1}^{k}n|\alpha_n|^2-\sum\limits_{n=1}^k\frac{1}{n}\right\rbrace
\end{align*}
and so from \eqref{Th-23.6}, we have
\begin{align}\label{Th-23.7}
\left|\frac{D^{k+1}f(0)\left(z_0^{k+1}\right)}{(k+1)!}-\xi\frac{D^kf(0)\left(z_0^k\right)}{k!} \right|\leq \exp\left\lbrace \sum\limits_{n=1}^{k}\frac{1}{n}|g_n-\xi^n|^2-\sum\limits_{n=1}^k\frac{1}{n}\right\rbrace.
\end{align}

By Corollary \cite[pp.178]{Duren-1983}, we can choose $\xi\in\partial \mathbb{U}$ so as to make the exponent in \eqref{Th-23.7} nonpositive, to obatin
\begin{align*}
	\left|\frac{1}{(k+1)!}D^{k+1}f(0)\left(z_0^{k+1}\right)-\xi\frac{1}{k!} D^kf(0)\left(z_0^k\right)\right|\leq 1.
\end{align*}
Therefore,
\begin{align}\label{Th-23.8}
	\displaystyle
	&\left|\;\left|\frac{1}{(k+1)!}D^{k+1}f(0)\left(z_0^{k+1}\right)\right|-\left|\frac{1}{k!} D^kf(0)\left(z_0^k\right)\right|\;\right|\\\leq& 	\left|\frac{1}{k+1}D^{(k+1)!}f(0)\left(z_0^{k+1}\right)-\xi\frac{1}{k!} D^kf(0)\left(z_0^k\right)\right|\leq 1\nonumber
\end{align}
for all $\xi\in\partial \mathbb{U}$.
Since $z_0=z/\|z\|$, using \eqref{Th-21.7}, we deduce from \eqref{Th-23.8} that
\begin{align*}
	\displaystyle
	&\left|\;\left|\frac{l_z\left(D^{k+2}F(0)\left(z^{k+2}\right)\right)}{(k+2)!||z||^{k+2}}\right|-\left|\frac{l_z\left(D^{k+1}F(0)\left(z^{k+1}\right)\right)}{(k+1)!||z||^{k+1}} \right|\;\right|\leq 1,\quad z\in\mathbb{B}\backslash\{0\}\nonumber
\end{align*}
for $k\geq 1$.\vspace{2mm}

To establish the sharpness of the bound, let us consider the function $F$ defined as in Theorem \ref{Th-22}.
Putting $z=ru\;(0<r<1)$ into \eqref{Th-22.11}, we get
\begin{align*}
\frac{l_z\left(D^{k+1}F(0)\left(z^{k+1}\right)\right)}{(k+1)!||z||^{k+1}}=k+1,\quad \frac{D^{k}F(0)\left(z^{k}\right)}{k!||z||^k}=k
\end{align*}
and so equality of \eqref{Th-23t} holds.
\end{proof}

\section{{\bf Coefficient theorems of quasi-convex mappings of type $B$ in complex Banach space}}\label{Sec-3}

Let $X$ and $Y$ be two complex Banach spaces and $\Omega\subset X$ be a domain.
Also let $f:\Omega\rightarrow Y$ be a holomorphic mapping. We say that $f$ is
\emph{convex} on $\Omega$ if $f$ is biholomorphic on $\Omega$ and
$f(\Omega)$ is a convex domain in $Y$.\vspace{1.2mm}
A normalized locally biholomorphic mapping $f\in\mathcal{H}(\mathbb{B})$ is said to be quasi-convex of type $B$ if (see \cite[Definition 2.6]{Ham2})
\begin{align}\label{C1}
\displaystyle
\Re\left\lbrace l_z\left([Df(z)]^{-1}\left(D^2f(z)(z,z)+Df(z)z\right)\right)\right\rbrace>0,\quad z\in\mathbb{B},~l_z\in T(z).
\end{align}
If $\mathbb{B}$ is the unit ball of $X$, we let $\mathcal{K}_B(\mathbb{B})$ denote the subset of $\mathcal{S}(\mathbb{B})$
consisting of normalized quasi-convex mappings of type $B$ from $\mathbb{B}$ into $X$.\vspace{1.2mm}

Convex mappings in higher dimensions have been extensively studied in recent literature. Comprehensive overviews can be found in the monographs by Gong \cite{Gong} and Kohr \& Liczberski  \cite{Kohr- Liczberski-1998}. For foundational work on this topic, see \cite{Ham4},
\cite{Roper-Suffridge-1999}-\cite{Su3}, \cite{Xu-Dl-Li}, \cite{Xu-Yang-Liu-Xu-2015}, \cite{Zhang-Liu-2002}.
\vspace{1.2mm}
By considering the biholomorphic mappings $F:\mathbb{B}\to X$ of the form $F(z)=f(z)z$ for $z\in\mathbb{B}$, where $f:\mathbb{B}\to \mathbb{C}$ is a holomorphic function satisfying $f(0)=1$, we begin by providing a sufficient condition on the coefficients of a function $F$ to be quasi-convex in a complex Banach space.
 \begin{theo}\label{Th-31}Let $f\in\mathcal{H}(\mathbb{B},\mathbb{C})$ with $f(0)=1$ and $F(z)=zf(z)\in\mathcal{S}(\mathbb{B})$. If
\begin{align*}	
\sum\limits_{n=2}^{\infty}n^2\left|l_{z}\left(\frac{D^{n}F(0)(z^{n})}{n!||z||^n}\right)\right|\leq 1
\end{align*}	
for $z\in\mathbb{B}\backslash\{0\}$ and $l_z\in T(z)$, then $F(z)=zf(z)\in\mathcal{K}_B(\mathbb{B})$.	
\end{theo}
In the next result, we show that the Bieberbach conjecture holds for any function $F(z)=zf(z) \in \mathcal{K}_B(\mathbb{B})$.

\begin{theo}\label{Th-32}Let $f\in\mathcal{H}(\mathbb{B},\mathbb{C})$ with $f(0)=1$ and $F(z)=zf(z)\in\mathcal{K}_B(\mathbb{B})$. Then for for $z\in\mathbb{B}\backslash\{0\}$ and $l_z\in T(z)$, we have
\begin{align*}
\left|l_{z}\left(\frac{D^{k}F(0)(z^{k})}{k!||z||^k}\right)\right|\leq 1
\end{align*}
and
\begin{align}\label{Th-32t}
n^2(n+1)^2\left|l_{z}\left(\frac{D^{n}F(0)(z^{n})}{n!||z||^n}\right)\right|^2\leq 4\left(1+\sum\limits_{k=2}^{n}k^3	\left|l_{z}\left(\frac{D^{k}F(0)(z^{k})}{k!||z||^k}\right)\right|^2\right).	
\end{align}
Equality holds for the function $F(z)=f(z)z$, where $f(z)=(1-l_u(z))^{-1}$, $z\in\mathbb{B}$ and $||u||=1$.
\end{theo}

The next result concerns the difference of coefficients for functions $F(z)=zf(z) \in \mathcal{K}_B(\mathbb{B})$.

 \begin{theo}\label{Th-33}Let $f\in\mathcal{H}(\mathbb{B},\mathbb{C})$ with $f(0)=1$ and $F(z)=zf(z)\in\mathcal{K}_B(\mathbb{B})$. Then
\begin{align*}\
\displaystyle
&\left|\;\left|\frac{l_z\left(D^{k+1}F(0)\left(z^{k+1}\right)\right)}{(k+1)!||z||^{k+1}}\right|-\left|\frac{l_z\left(D^{k}F(0)\left(z^{k}\right)\right)}{k!||z||^{k}} \right|\;\right|\leq 1,\quad z\in\mathbb{B}\backslash\{0\}
\end{align*}
for $k=2,3$.
\end{theo}

\subsection{{Auxiliary lemma}}

The following Lemma is due to Liu and Liu \cite{Liu-Liu-2007}.
\begin{lem}\label{Lem2} {\em{\cite{Liu-Liu-2007}}}
Suppose that $f\in \mathcal{S}(\mathbb{U})$. Then $F$ defined by $F(z)=\frac{f(l_u(z))}{l_u(z)}z$, where $z\in\mathbb{B}$, $||u||=1$ belongs $\mathcal{K}_B(\mathbb{B})$ if and only if $f\in\mathcal{K}(\mathbb{U})$.
\end{lem}

\begin{proof}[\bf Proof of Theorem \ref{Th-31}]
By the given condition, we have $f\in\mathcal{H}(\mathbb{B},\mathbb{C})$ with $f(0)=1$ and $F(z)=zf(z)\in\mathcal{S}(\mathbb{B})$.	
Let $z\in \mathbb{B}\backslash\{0\}$ be fixed and let $z_0=z/\|z\|$. We define the function $g:\mathbb{U}\to \mathbb{C}$ as follows
\begin{align}\label{Th-31.0}	
\displaystyle
g(\zeta)=\begin{cases}
\frac{l_{z}\left([DF(\zeta z_0)]^{-1}\left(D^2 F(\zeta z_0)(\zeta^2z_0^2)+DF(\zeta z_0) \zeta z_0\right)\right)}{\zeta},& \zeta\neq 0,\\[8pt]
1,& \zeta=0.
\end{cases}
\end{align}
Let $\phi(z)=D^2 F(z)(z^2)+DF(z)z$.		
Since $F(z)$ is normalized, it has the power series expansion
\begin{align}\label{Th-31.2}
\displaystyle
F(z)=z+\sum\limits_{k=2}^{\infty}\frac{1}{k!}D^k F(0)(z^k).
\end{align}
Clearly
\begin{align}\label{Th-31.3}
\displaystyle
D F(z)=I+\sum\limits_{k=2}^{\infty}\frac{k}{k!}D^k F(0)(z^{k-1})
\end{align}
and
\begin{align}\label{Th-31.4}
\displaystyle
D^2 F(z)=\sum\limits_{k=2}^{\infty}\frac{k(k-1)}{k!}D^k F(0)(z^{k-2})
\end{align}
Using \eqref{Th-31.3} and \eqref{Th-31.4}, we get
\begin{align*}
\phi(z)=z+\sum\limits_{k=2}^{\infty}\frac{k^2}{k!}D^k F(0)(z^{k})
\end{align*}
and so
\begin{align}\label{Th-31.5}
l_{z_0}(\phi(\zeta z_0))=\zeta+\sum\limits_{k=2}^{\infty}k^2a_k\zeta^k,
\end{align}	
where
\begin{align}\label{Th-31.5a}
a_k=l_{z_0}\left(\frac{1}{k!}D^3F(0)(z_0^3)\right).
\end{align}	
On the other hand, from \eqref{Th-21.8} and \eqref{Th-21.9}, we obtain
\begin{align}\label{Th-31.6}
\psi(\zeta)=f(\zeta z_0)+D f(\zeta z_0)\zeta z_0=1+\sum\limits_{k=1}^{\infty}(k+1)a_{k+1}\zeta^k=1+\sum\limits_{k=2}^{\infty}ka_{k}\zeta^{k-1}
\end{align}
Again \eqref{Th-21.1} gives
\begin{align}\label{Th-31.7}
\displaystyle
[D F(z)]^{-1}\phi(z)=\frac{\phi(z)}{f(z)+D f(z)z}.
\end{align}	
In view of \eqref{Th-31.0} and using \eqref{Th-31.5}, \eqref{Th-31.6} to \eqref{Th-31.7}, we obtain
\begin{align}\label{Th-31.8}
\displaystyle
g(\zeta)=\frac{1+\sum\limits_{k=2}^{\infty}k^2a_k\zeta^{k-1}}{1+\sum\limits_{k=2}^{\infty}ka_{k}\zeta^{k-1}}.	
\end{align}
Now applying \eqref{Th-31.6} and \eqref{Th-31.8}, we have
\begin{align*}
\left|g(\zeta)\psi(\zeta)-\psi(\zeta)\right|-|\psi(\zeta)|=&\left|\sum\limits_{k=2}^{\infty}(k^2-k)a_k\zeta^{k-1}\right|-\left|1+\sum\limits_{k=2}^{\infty}ka_{k}\zeta^{k-1}\right|\\<&
\sum\limits_{k=2}^{\infty}(k^2-k)|a_k|-\left(1-\sum\limits_{k=2}^{\infty}k|a_{k}|\right)=
\sum\limits_{k=2}^{\infty}k^2|a_k|-1\leq 0
\end{align*}
and so for $\zeta\in\mathbb{U}\backslash\{0\}\}$, we have $\left|g(\zeta)-1\right|<1.$ This shows that $g(\zeta)$ lies inside a circle with center at $(1,0)$ and radius $1$ and so $\Re\{g(\zeta)\}>0$ for all $\zeta\in\mathbb{U}\backslash\{0\}$, i.e.,
\begin{align}\label{Th-31.9}
\Re\left\lbrace \frac{l_{z_0}\left([D F(\zeta z_0)]^{-1}\phi(\zeta z_0)\right)}{\zeta}\right\rbrace>0
\end{align}
for all $\zeta\in\mathbb{U}\backslash\{0\}$. Substituting $\zeta=||z||$ into \eqref{Th-31.9} and using the fact that
$l_{\alpha z}=\frac{|\alpha|}{\alpha}l_z$
for each $\alpha\in\mathbb{U}\backslash\{0\}$, we obtain
\begin{align*}
	\Re\left\lbrace \frac{l_{z}\left([D F(z)]^{-1}\left(D^2 F(z)(z^2)+DF(z) z\right)\right)}{||z||}\right\rbrace>0,\quad z\in\mathbb{B}\backslash\{0\}, \quad l_z\in T(z).
\end{align*}

Consequently \eqref{C1} yields $F\in \mathcal{K}_B(\mathbb{B})$.
\end{proof}

\begin{proof}[\bf Proof of Theorem \ref{Th-32}]
By the given condition, we have $F\in\mathcal{K}_B(\mathbb{B})$ and so using \eqref{C1}, we must have
\begin{align*}
\displaystyle
\Re\left\lbrace \frac{l_z\left([DF(z)]^{-1}\left(D^2 F(z)(z^2)+DF(z) z\right)\right)}{||z||}\right\rbrace>0, \quad z\in\mathbb{B}\backslash\{0\},\quad l_z\in T(z),
\end{align*}
where $\phi(z)=D^2 F(z)(z^2)+DF(z) z$.	
Let $z\in \mathbb{B}\backslash\{0\}$ be fixed and let $z_0=z/\|z\|$. We consider the function $g:\mathbb{U}\to \mathbb{C}$ defined by \eqref{Th-31.0}.
Clearly $g\in\mathcal{H}(\mathbb{U})$ and $g(0)=1$.
Therefore for $\zeta\in\mathbb{U}\backslash\{0\}$, we have
\begin{align*}
\displaystyle
g(\zeta)=\frac{l_{\zeta z_0}\left([DF(\zeta z_0)]^{-1}\left(D^2 F(\zeta z_0)(\zeta^2z_0^2)+DF(\zeta z_0) \zeta z_0\right)\right)}{||\zeta z_0||}.
\end{align*}	
Consequently $\Re\{g(\zeta)\}>0$ for all $\zeta\in\mathbb{U}$ and so we may assume that
\begin{align}\label{Th-32.1}
\displaystyle
g(\zeta)=1+p_1\zeta+p_2\zeta^2+\ldots
\end{align}
Now using \eqref{Th-31.8} and \eqref{Th-32.1}, we get
\begin{align}\label{Th-32.2}
\displaystyle
\left(1+\sum\limits_{k=1}^{\infty}p_k\zeta^k\right)\left(1+\sum\limits_{k=2}^{\infty}ka_{k}\zeta^{k-1}\right)=1+\sum\limits_{k=2}^{\infty}k^2a_k\zeta^{k-1},	
\end{align}
where $a_k$ is defined by \eqref{Th-31.5a}.
Equating coefficients on the both sides of \eqref{Th-32.2}, we get
\begin{align}\label{Th-32.3}
k^2a_k=ka_k+(k-1)p_1a_{k-1}+\ldots+2p_{k-2}a_2+p_{k-1}.
\end{align}
We now use mathematical induction and suppose $|a_i|\leq 1$ for $i=2,3,\ldots, k-1$.
Applying Lemma \ref{Lem1} to \eqref{Th-32.3}, we get
\begin{align*}
(k^2-k)|a_k|\leq 2\sum\limits_{i=1}^{k-1}i=k(k-1),
\end{align*}
which shows that $|a_k|\leq 1$.
Then by induction, we have the first inequality
\begin{align*}
	\left|l_{z_0}\left(\frac{1}{k!}D^{k}F(0)(z_0^{k})\right)\right|\leq 1.
\end{align*}
Since $\Re\{g(\zeta)\}>0$, there exists a Schwarz function $\omega(\zeta)$ with $\omega(\zeta)=\sum_{k=1}^{\infty}\omega_k\zeta^k$ such that $g(\zeta)=(1+\omega(\zeta))/(1-\omega(\zeta))$ and so from \eqref{Th-31.8} and \eqref{Th-32.1}, we get
\begin{align}\label{Th-32.4}
\left(2+\sum\limits_{k=2}^{\infty}k(k+1)a_{k} \zeta^{k-1} \right)\omega(\zeta)=\sum\limits_{k=2}^{\infty}k(k-1)a_{k}\zeta^{k-1}.
\end{align}
Equating coefficients on the both sides of \eqref{Th-32.4}, one can easily get
\begin{align*}
	2\omega_{k-1}+6a_2\omega_{k-2}+\ldots+k(k-1)a_{k-1}\omega_1=k(k-1)a_{k}.
\end{align*}
This shows that the coefficient $a_{k}$ depends only on $a_2$, $a_3\ldots$,$a_{k-1}$. Therefore we can rewrite \eqref{Th-32.4} as follows
\begin{align}\label{Th-32.5}
	\left(2+\sum\limits_{k=2}^{n-1}k(k+1)a_{k} \zeta^{k-1} \right)\omega(\zeta)=\sum\limits_{k=2}^{n}k(k-1)a_{k}\zeta^{k-1}+\sum\limits_{k=n+1}^{\infty}b_k\zeta^{k-1}
\end{align}
for some $b_k$.
Let $\zeta=re^{i\theta}$. Then from \eqref{Th-32.5}, we get
\begin{align}\label{Th-32.6}
\left(2+\sum\limits_{k=2}^{n-1}k(k+1)a_{k} r^{k-1}e^{i(k-1)\theta} \right)\omega(re^{i\theta})=&\sum\limits_{k=2}^{n}k(k-1)a_{k}r^{k-1}e^{i(k-1)\theta}\\&+\sum\limits_{k=n+1}^{\infty}b_kr^{k-1}e^{i(k-1)\theta}\nonumber
\end{align}
and
\begin{align}\label{Th-32.7}
\left(2+\sum\limits_{k=2}^{n-1}k(k+1)\ol a_{k} r^{k-1}e^{-i(k-1)\theta} \right)\omega(re^{-i\theta})=&\sum\limits_{k=2}^{n}k(k-1)\ol a_{k}r^{k-1}e^{-i(k-1)\theta}\\&+\sum\limits_{k=n+1}^{\infty}\ol b_kr^{k-1}e^{-i(k-1)\theta}\nonumber.
\end{align}

Note that $|\omega(\zeta)|<1$ for all $\zeta\in\mathbb{U}$. We first multiply \eqref{Th-32.6} and \eqref{Th-32.7} and then integrating with respect to $\theta$ from $0$ to $2\pi$, we can find that
\begin{align*}
\sum\limits_{k=2}^{n}k^2(k-1)^2|a_{k}|^2r^{2(k-1)}+\sum\limits_{k=n+1}^{\infty}|b_k|^2r^{2(k-1)}<4+\sum\limits_{k=2}^{n-1}k^2(k+1)^2|a_{k}|^2r^{2(k-1)}.
\end{align*}
Letting $r\to 1$, we obtain
\begin{align*}
	\sum\limits_{k=2}^{n}k^2(k-1)^2|a_{k}|^2\leq 4+\sum\limits_{k=2}^{n-1}k^2(k+1)^2|a_{k}|^2,
\end{align*}
i.e.,
\begin{align*}
n^2(n-1)^2|a_{n}|^2\leq 4\left(1+\sum\limits_{k=2}^{n-1}k^3|a_{k}|^2\right).
\end{align*}
Now adding $4n^3|a_{n}|^2$ on the both sides and then using \eqref{Th-31.5a}, we get the required conclusion \eqref{Th-32t}.\vspace{1.9mm}

{\bf Sharpness part:} We have $F(z)=\frac{f(l_u(z))}{l_u(z)}z$, where $f(l_u(z))=l_u(z)(1-l_u(z))^{-1}$, $z\in\mathbb{B}$ and $||u||=1$. Clearly $f\in \mathcal{K}(\mathbb{U})$.
According to Lemma \ref{Lem2}, it is clear that $F\in\mathcal{K}_B(\mathbb{B})$. Note that
$F(z)=\sum_{n=0}^{\infty}(l_u(z))^{n}\;z$, $z\in\mathbb{B}$.
A straightforward computation shows that
\begin{align*}
\frac{D^{k}F(0)\left(z^{k}\right)}{k!}=(l_u(z))^{k-1} z.
\end{align*}
Taking $z=ru\;(0<r<1)$, we get
\begin{align*}
\frac{D^{k}F(0)\left(z^{k}\right)}{k!||z||^k}=1
\end{align*}
and so equality of \eqref{Th-32t} holds.
\end{proof}

\begin{proof}[\bf Proof of Theorem \ref{Th-33}]
By the given condition, $F\in\mathcal{K}_B(\mathbb{B})$. Let $z\in \mathbb{B}\backslash\{0\}$ be fixed. Suppose $z_0=z/\|z\|$. We consider the function $g:\mathbb{U}\to \mathbb{C}$ defined by \eqref{Th-31.0}. Clearly
$g\in\mathcal{H}(\mathbb{U})$ with $g(0)=1$ and $\Re\{g(\zeta)\}>0$ for all $\zeta\in\mathbb{U}$.
Now from \eqref{Th-32.2}, we have
\begin{align}\label{Th-33.1}
	\displaystyle
	\left(1+\sum\limits_{k=1}^{\infty}p_k\zeta^k\right)\left(1+\sum\limits_{k=2}^{\infty}ka_{k}\zeta^{k-1}\right)=1+\sum\limits_{k=2}^{\infty}k^2a_k\zeta^{k-1},	
\end{align}
where $a_k$ is defined by \eqref{Th-31.5a}.		
Clearly from \eqref{Th-33.1}, we find that
\begin{align}\label{Th-33.2}
\displaystyle
a_2=\frac{l_z(D^2F(0)(z^2))}{2!\|z\|^2}=\frac{1}{2}p_1,
\end{align}
\begin{align}\label{Th-33.3}
\displaystyle
a_3=\frac{l_{z}(D^3F(0)(z^3))}{3!\|z\|^3}=\frac{1}{6}\left(p_2+p_1^2\right)
\end{align}
and
\begin{align}\label{Th-33.4}
\displaystyle
a_4=\frac{l_{z}(D^4F(0)(z^4))}{4!\|z\|^4}=\frac{1}{12}\left(p_3+\frac{3}{2}p_1p_2+\frac{1}{2}p_1^3\right).
\end{align}
	
We now consider the following two cases.
\medskip
	
{\bf Case 1.} Let $k=2$. Now with the help of \eqref{Th-33.2} and \eqref{Th-33.3}, we find that
\begin{align*}
\displaystyle
\left|\;\left|\frac{l_{z}(D^2F(0)(z^2))}{2!\|z\|^2}\right|-\left|\frac{l_z(D^3F(0)(z^3))}{3!\|z\|^3}\right|\;\right|\leq&
\left|\frac{1}{6}\left|p_2+p_1^2\right|-\frac{1}{2}|p_1|\right|\\\leq& \max\left\lbrace \frac{1}{6}|p_2+p_1^2|,\frac{1}{2}|p_1|\right\rbrace
\end{align*}
and so by Lemma \ref{Lem1}, we have
\begin{align*}
\displaystyle
\left|\;\left|\frac{l_{z}(D^3F(0)(z^3))}{3!\|z\|^3}\right|-\left|\frac{l_z(D^2F(0)(z^2))}{2!\|z\|^2}\right|\;\right|\leq 1.
\end{align*}
\medskip

{\bf Case 2.} Let $k=3$.  Now with the help of \eqref{Th-33.3} and \eqref{Th-33.4}, we find that
\begin{align*}
\displaystyle
&\left|\;\left|\frac{l_{z}(D^4F(0)(z^4))}{4!\|z\|^4}\right|-\left|\frac{l_z(D^3F(0)(z^3))}{3!\|z\|^3}\right|\;\right|\\\leq&
\left|\frac{1}{12}\left|p_3+\frac{3}{2}p_1p_2+\frac{1}{2}p_1^3\right|-\frac{1}{6}\left|p_2+p_1^2\right|\right|\\\leq& \max\left\lbrace \frac{1}{12}\left|p_3+\frac{3}{2}p_1p_2+\frac{1}{2}p_1^3\right|, \frac{1}{6}|p_2+p_1^2|\right\rbrace
\end{align*}
and so by Lemma \ref{Lem1}, we have
\begin{align*}
\left|\;\left|\frac{l_{z}(D^4F(0)(z^4))}{4!\|z\|^4}\right|-\left|\frac{l_z(D^3F(0)(z^3))}{3!\|z\|^3}\right|\;\right|\leq 1.
\end{align*}	
	
\end{proof}

\vspace{4mm}
\noindent\textbf{Conflict of interest:} The authors declare that there is no conflict  of interest regarding the publication of this paper.\vspace{1.2mm}

\noindent {\bf Funding:} Not Applicable.

\noindent\textbf{Data availability statement:}  Data sharing not applicable to this article as no datasets were generated or analysed during the current study.\vspace{1.2mm}


\end{document}